\documentclass{article}
\usepackage{amsmath,amssymb}
\title{Blow-up of solutions to a Dirichlet problem for the discrete semi-linear heat equation}
\author{Keisuke Matsuya\\Graduate School of Mathematical Sciences, University of Tokyo,\\ Komaba 3-8-1, Meguro, Tokyo
153-8914, JAPAN}
\newtheorem{theorem}{Theorem}
\newtheorem{proposition}{Proposition}[section]
\newtheorem{lemma}{Lemma}[section]
\newtheorem{definition}{Definition}[section]
\newtheorem{proof}{Proof}

\begin{document}
\maketitle
\section{Introduction}
In this paper, we consider the following partial difference equation with prescribed initial and boundary conditions:
\begin{equation}\label{dheq}
\begin{cases}
\displaystyle{f^{s+1}_{\vec{n}} = \frac{g^s_{\vec{n}}}{\{1-\alpha\delta(g^s_{\vec{n}})^\alpha\}^{1/\alpha}}\ (s\in\mathbb{Z}_{\ge0},\ \vec{n}\in\Omega_{\text{D}}^\circ)},\\
f^0_{\vec{n}}=a_{\vec{n}}\ge0,\not\equiv0\ (\vec{n}\in\Omega_{\text{D}}),\\
f^s_{\vec{n}}=0\ (s\in\mathbb{Z}_{\ge0}\ \vec{n}\in\partial\Omega_{\text{D}}),
\end{cases}
\end{equation}
$\Omega_{\text{D}}$ is a bounded subset of $\mathbb{Z}^d$,\ $\partial \Omega_{\text{D}}$ is the boundary of $\Omega_{\text{D}}$, $\Omega_{\text{D}}^\circ$ is the interior of $\Omega_{\text{D}}$, (namely $\Omega_{\text{D}}^\circ:=\Omega_{\text{D}} \setminus \partial \Omega_{\text{D}}$),\ $f^s_{\vec{n}}:=f(s,\vec{n}),\ s\in\mathbb{Z}_{\ge0},\ \vec{n}\in\Omega_{\text {D}}$.
Moreover, we take\ $\alpha,\delta>0$ and $g^s_{\vec{n}}$ define as:
\begin{equation*}
g^s_{\vec{n}}:=\sum^d_{k=1}\frac{f^s_{\vec{n}+\vec{e}_k}+f^s_{\vec{n}-\vec{e}_k}}{2d},
\end{equation*}
where $\vec{e}_k$ is the unit vector whose $k$-th component is 1 and the others are 0.

The difference equation in\eqref{dheq} is investigated \cite{MT} as a discretization of the following semi-linear heat equation:
\begin{equation}\label{heq}
\displaystyle{\frac{\partial f}{\partial t}=\Delta f + f^{1+\alpha}},
\end{equation}
where $f:=f(t,\vec{x}),\ t\ge0,\ \vec{x}\in\Omega_{\text{C}}\subset\mathbb{R}^d$ and $\Delta$ is a $d$-dimensional Laplacian.

Solutions of \eqref{heq} are not necessarily bounded for all $t\ge0$.
In general, if there exists a finite time $T>0$ for which the solution of \eqref{heq} in $(t,\vec{x})\in[0,T)\times\Omega_{\text{C}}$ satisfies
\begin{equation*}
\limsup\limits_{t\to T-0}{\| f(t,\cdot)\|_{L^{\infty}}}=\infty,
\end{equation*}
where
\begin{equation*}
\| f(t,\cdot)\|_{L^\infty}:=\sup\limits_{\vec{x}\in\Omega_{\text{C}}}{|f(t,\vec{x})|},
\end{equation*}
then we say that the solution of \eqref{heq} blows up at time $T$.

The Cauchy problem for \eqref{heq} has been studied and a critical exponent which characterises the blow-up of the solutions for \eqref{heq} has been discovered and studied by Fujita and et al.\cite{F,H,KST,W}\\
In fact, the difference equation \eqref{dheq} has similar characteristics to the critical exponent known from the continuous case.

Considering \eqref{heq} on $[0,T)\times\Omega_{\text{C}}$ with the following initial and boundary conditions
\begin{equation}\label{ib}
\begin{cases}
f(0,\vec{x})=a(\vec{x})\ge0,\not\equiv0\ (\vec{x}\in\Omega_{\text{C}}),\\
f(t,\vec{x})=0\ (t\ge0,\ \vec{x}\in\partial\Omega_{\text{C}}),
\end{cases}
\end{equation}
where $\Omega_{\text{C}}$ is a bounded subset of $\mathbb{R}^d$,
the following theorem can be shown to hold.
\begin{theorem}[\cite{F}]
The solution of \eqref{heq} with initial and boundary conditions \eqref{ib} does not blow up at any finite time for sufficiently small initial conditions $a(\vec{x})$.
\end{theorem}

In this article, we show that \eqref{dheq} has a property similar to theorem 1.
In section 2, we define the blow-up of solutions for \eqref{dheq} and state the main theorem which is a discrete analogue of theorem 1.
This theorem is proved in section 3.
\section{Main theorem}
First, we define the blow-up of solutions for \eqref{dheq}.
Because of the term $\displaystyle \{1-\alpha\delta(g^s_{\vec{n}})\}^{1/\alpha}$, when $g^s_{\vec{n}} \to (\alpha\delta)^{-1/\alpha}-0$, then $f^{s+1}_{\vec{n}} \to +\infty$.
This behaviour may be regarded as an analogue of the blow up of solutions for the semi-linear heat equation.
Thus we define a global solution of \eqref{dheq} as follows.
\begin{definition}
Let $f^s_{\vec{n}}$ be a solution to \eqref{dheq}.

When there exists an $s_0 \in \mathbb{Z}_{{}\ge0}$ such that $g^{s}_{\vec{n}} \le (\alpha\delta)^{-1/\alpha}$ for all $s < s_0$ and $\vec{n} \in \Omega_{\text{D}}$, and when there exists $\vec{n}_0 \in \Omega_{\text{D}}$ such that $g^{s_0}_{\vec{n}_0} \ge (\alpha\delta)^{-1/\alpha}$, then we say that the solution $f^s_{\vec{n}}$ blows up at time $s_0$.
\end{definition}
The following theorem is the main theorem of this paper.
\begin{theorem}
For $\Omega_{\text{D}}=\{\vec{n}=(n_1,\cdots,n_d)\in\mathbb{Z}^d|0\le n_k\le N_k\ (k=1,\cdots,d)\}$, the solution of \eqref{dheq} does not blow up at any finite time with sufficiently small initial condition $a_{\vec{n}}$.
\end{theorem}
\section{Proof of the theorem}
To prove the theorem, we make use of a comparison theorem.

First, to simplify the equations, we take the scaling $(\alpha\delta)^{1/\alpha}f^s_{\vec{n}} \to f^s_{\vec{n}}$ which changes the difference equation in \eqref{dheq} to
\begin{equation*}
\displaystyle{f^{s+1}_{\vec{n}} = \frac{g^s_{\vec{n}}}{\{1-(g^s_{\vec{n}})^\alpha\}^{1/\alpha}}}
\end{equation*}
Now, we construct a majorant solution.
Let
\begin{equation}\label{dslh2}
\displaystyle{\widehat{M}(h_{\vec{n}}):=\frac{1}{2d}\sum^d_{k=1}{(h_{\vec{n}+\vec{e}_k}+h_{\vec{n}-\vec{e}_k})}}.
\end{equation}
We denote by $h^s_{\vec{n}}$ the solution to the initial and boundary condition problem of the linear partial difference equation
\begin{equation}\label{dlheq}
\begin{cases}
h^{s+1}_{\vec{n}} = \widehat{M}(h^s_{\vec{n}})\ (s\in\mathbb{Z}_{\ge0},\ \vec{n}\in\Omega_{\text{D}}^\circ) \\
h^0_{\vec{n}} = a_{\vec{n}}\ (\vec{n}\in\Omega_{\text{D}}),\\
h^s_{\vec{n}} =0\ (s\in\mathbb{Z}_{\ge0},\ \vec{n}\in\partial\Omega_{\text{D}}).
\end{cases}
\end{equation}
The majorant solution is $\bar{f}^s_{\vec{n}}$ defined as follow:
\begin{equation}\label{ssol}
\displaystyle{\bar{f}^s_{\vec{n}} := \frac{h^s_{\vec{n}}}{\left\{1-\sum\limits^s_{k=0}{|m_k|^\alpha}\right\}^{1/\alpha}}},
\end{equation}
where $m_s$ is defined in terms of \eqref{dlheq} as
\begin{equation}\label{ms}
m_s := \max_{\vec{n}\in\Omega_{\text{D}}^\circ}{h^s_{\vec{n}}}.
\end{equation}
\begin{lemma}
When $\bar{f}^s_{\vec{n}}$ exists at $s$, for all $\vec{n} \in \mathbb{Z}^d$, namely when
\begin{equation*}
1-\sum^s_{k=0}{|m_k|^\alpha}>0
\end{equation*}
holds, the solution of \eqref{dheq} does not blow up at any time $s$ and moreover satisfies
\begin{equation}\label{ineq}
\bar{f}^s_{\vec{n}} \geq f^s_{\vec{n}}.
\end{equation}
\end{lemma}
\begin{proof}
We precede by induction on $s$.
When $s=0$, by the definition of the initial and boundary condition problem,  $f^0_{\vec{n}}$ exists and \eqref{ineq} holds because
\begin{equation*}
\displaystyle{\bar{f}^0_{\vec{n}} = \frac{h^0_{\vec{n}}}{\{1-|m_0|^\alpha\}^{1/\alpha}} \geq h^0_{\vec{n}} = f^0_{\vec{n}}}.
\end{equation*}
Suppose that the statement is true up to $s=s_0$ and that $\bar{f}^{s_0+1}_{\vec{n}}$ exists.
When $\bar{f}^{s_0+1}_{\vec{n}}=0$, we have that
\begin{eqnarray*}
\bar{f}^{s_0+1}_{\vec{n}}=0 &\  \Longleftrightarrow \ &h_{\vec{n}}^{s_0+1}=0\\
&\  \Longleftrightarrow \ &\widehat{M}(h_{\vec{n}}^{s_0})=0\\
&\  \Longleftrightarrow \ &h_{\vec{n}\pm \vec{e}_k}^{s_0}=0 \quad (k=1,2,\ldots, d)\\
&\  \Longleftrightarrow \ &\bar{f}_{\vec{n}\pm \vec{e}_k}^{s_0}=0 \quad (k=1,2,\ldots, d)\\
&\  \Longrightarrow \ &f_{\vec{n}\pm \vec{e}_k}^{s_0}=0 \quad (k=1,2,\ldots, d)\\
&\  \Longleftrightarrow \ &g_{\vec{n}}^{s_0}=0 \\
&\  \Longleftrightarrow \ &f_{\vec{n}}^{s_0+1}=0.
\end{eqnarray*}
Hence  \eqref{ineq} holds.

When $\bar{f}^{s_0+1}_{\vec{n}}>0$, if $g_{\vec{n}}^{s_0} = 0 $, then $f_{\vec{n}}^{s_0+1}=0$ and the statement is true.
Otherwise
\begin{eqnarray*}
0<(\bar{f}^{s_0+1}_{\vec{n}})^{-\alpha} &=& \frac{1-\sum\limits^{s_0+1}_{k=0}{|m_k|^\alpha}}{(h^{s_0+1}_{\vec{n}})^\alpha} = \frac{1-\sum\limits^{s_0}_{k=0}{|m_k|^\alpha}}{(h^{s_0+1}_{\vec{n}})^\alpha} - \left|\frac{m_{s_0+1}}{h^{s_0+1}_{\vec{n}}}\right|^\alpha \\
&\leq& \frac{1-\sum\limits^{s_0}_{k=0}{|m_k|^\alpha}}{\left\{\widehat{M}(h^{s_0}_{\vec{n}})\right\}^\alpha} - 1 =\frac{1}{\left\{\widehat{M}(\bar{f}^{s_0}_{\vec{n}})\right\}^\alpha} - 1 \\
&\leq& (g^{s_0}_{\vec{n}})^{-\alpha} - 1.
\end{eqnarray*}
From \eqref{dslh2}, $(g^{s_0}_{\vec{n}})^{-\alpha} - 1=(f^{s_0+1}_{\vec{n}})^{-\alpha}$ and we find
 $(\bar{f}^{s_0+1}_{\vec{n}})^{-\alpha} \leq (f^{s_0+1}_{\vec{n}})^{-\alpha}$,  i.e. $f^{s_0+1}_{\vec{n}} \leq \bar{f}^{s_0+1}_{\vec{n}}$.
Thus, from the induction hypothesis, the statement  is true for any non-negative integer $s$.
\end{proof}
From this lemma, by proving that $1-\sum^s_{k=0}{|m_k|^\alpha}>0$ for all $s\in\mathbb{Z}_{\ge0}$ with sufficiently small initial condition in \eqref{dheq}, one can complete the proof of the main theorem.

The solution of \eqref{dlheq} is 
\begin{equation*}
h^{s}_{\vec{n}} = \displaystyle{\sum\limits_{\vec{n}^{\prime}\in\Omega^\circ_{\text{D}}}{\left\{B_{\vec{n}^{\prime}}(c_{\vec{n}^{\prime}})^{s}\prod\limits_{k=1}^{d}{\sin{\left(\frac{n^{\prime}_k\pi}{N_k}n_k\right)}}\right\}}},
\end{equation*}
where $\vec{n}:=(n_1,\cdots,n_d),\ \vec{n}^\prime:=(n_1^\prime,\cdots,n_d^\prime),\ c_{\vec{n}^{\prime}}:=\sum^d_{k=1}{\frac{1}{d}\cos{(n^{\prime}_k\pi/N_k)}}$ and $B_{\vec{n}^\prime}$ are constants that satisfy $h^0_{\vec{n}}=a_{\vec{n}}$.
The following proposition concerning $B_{\vec{n}}$ can be proven.
\begin{proposition}
If the initial condition of \eqref{dlheq} $a_{\vec{n}}$ is fixed, $B_{\vec{n}}$ are determined uniquely.
\end{proposition}
\begin{proof}
This property is proved by induction on $d$.

When $d=1$, put $N:=N_1$.
Solving $N-1$ linear equations with $N-1$ unknowns:\ $a_{n^\prime}=\sum^{N-1}_{n=1}{B_n\sin{\left(\frac{n\pi}{N}n^\prime\right)}}\ (n^\prime=1,\cdots,N-1)$, the $B_n$ are determined.
If $N-1$ vectors $\left(\sin{\frac{n\pi}{N}},\cdots,\sin{\frac{n(N-1)\pi}{N}}\right)\ (n=1,\cdots,N-1)$ are linearly independent, then the $B_n$ are determined uniquely.
On the other hand, these $N-1$ vectors are eigenvectors of the following $(N-1) \times (N-1)$ matrix:
\begin{equation*}
\begin{pmatrix}
	0 & 1 & 0 & \ldots & 0 \\
	1 & 0 & \ddots & \ddots & \vdots \\
	0 & \ddots & \ddots & \ddots & 0 \\
	\vdots & \ddots & \ddots & 0 & 1 \\
	0 & \ldots & 0 & 1 & 0
\end{pmatrix}.
\end{equation*}
All eigenvector are linearly independent so that the $B_n$ are determined uniquely.

Suppose that the statement is true up to $d=d_0-1$.
Now we consider the case of $d=d_0$.
\begin{eqnarray*}
a_{\vec{n}^\prime} &=& \displaystyle{\sum\limits_{\vec{n}\in\Omega^\circ_{\text{D}}}{B_{\vec{n}}\prod\limits_{k=1}^{d_0}{\sin{\left(\frac{n_k\pi}{N_k}n^\prime_k\right)}}}}\\
&=& \displaystyle{\sum\limits_{n_{d_0}=1}^{N_{d_0}-1}{\sin{\left(\frac{n_{d_0}\pi}{N_{d_0}}n^\prime_{d_0}\right)}}\sum{B_{\vec{n}}\prod\limits_{k=1}^{d_0-1}{\sin{\left(\frac{n_k\pi}{N_k}n^\prime_k\right)}}}}
\end{eqnarray*}
If\ $n_1,\cdots,n_{d_0-1}$ are fixed, then each $\sum B_{\vec{n}}\prod^{d_0-1}_{k=1}{\sin{\left(\frac{n_k\pi}{N_k}n_k\right)}}$ is determined uniquely from the case of $d=1$.
Because of the induction hypothesis, the $B_{\vec{n}}$ are also determined uniquely.
Thus, the statement is true for any $d$.
\end{proof}
Now we estimate the infinite series $\sum^\infty_{k=0}{|m_k|^\alpha}$.
Take $B:=\max_{\vec{n}}|B_{\vec{n}}|$.
If one lets $\max_{\vec{n}}|a_{\vec{n}}|$ be small, $B$ also becomes small.
We consider three cases $\alpha\le1,\ \alpha >1$.

When $\alpha\le1$, we obtain
\begin{eqnarray*}
\displaystyle{\sum\limits_{k=0}^{\infty}{|m_k|^{\alpha}}} &\le& \displaystyle{\sum\limits_{k=0}^{\infty}\left(B\sum\limits_{\vec{n}\in\Omega^\circ_{\text{D}}}{|c_{\vec{n}}|^k}\right)^{\alpha}}\\
&\le&\displaystyle{\sum\limits_{k=0}^{\infty}B^\alpha\sum\limits_{\vec{n}\in\Omega^\circ_{\text{D}}}{|c_{\vec{n}}|^{k\alpha}}}\\
&=& B^\alpha\displaystyle{\sum\limits_{\vec{n}\in\Omega^\circ_{\text{D}}}{\frac{1}{1-|c_{\vec{n}}|^{\alpha}}}} < \infty.
\end{eqnarray*}
We used the inequality $(x+y)^\alpha \le x^\alpha+y^\alpha\ (x,y\ge0)$ in the second line.
The inequality above implies that $\sum^\infty_{k=0}{|m_k|^\alpha}$ can take an arbitrarily small value, if one lets the value of $B$ small.
Thus, $\sum^\infty_{k=0}{|m_k|^\alpha}<1$ with sufficiently small initial condition in \eqref{dlheq} and the statement of theorem 2 holds by lemma 3.1.

When $\alpha>1$, since $|c_{\vec{n}}|<1\ (\vec{n}\in\Omega_{\text{D}}^\circ)$, $|c_{\vec{n}}|^s\to0\ (s\to\infty)$ for all $\vec{n}\in\Omega_{\text{D}}^\circ$.
Thus, there exists $s_0\in\mathbb{Z}_{\ge0}$ such that $\sum_{\vec{n}\in\Omega_{\text{D}}^\circ}{|c_{\vec{n}}|^s}<1\ (s \ge s_0)$.
Now we get
\begin{eqnarray*}
\displaystyle{\sum\limits_{k=0}^{\infty}{|m_k|^{\alpha}}} &=& \displaystyle{\sum\limits_{k=0}^{s_0-1}{|m_k|^{\alpha}}} + \displaystyle{\sum\limits_{k=s_0}^{\infty}{|m_k|^{\alpha}}}\\
&\le& \displaystyle{\sum\limits_{k=0}^{s_0-1}{|m_k|^{\alpha}}} + \displaystyle{\sum\limits_{k=s_0}^{\infty}B^{\alpha}\left(\sum\limits_{\vec{n}\in\Omega^\circ_{\text{D}}}{|c_{\vec{n}}|^k}\right)^{\alpha}}\\
&\le& \displaystyle{\sum\limits_{k=0}^{s_0-1}{|m_k|^{\alpha}}} + \displaystyle{\sum\limits_{k=s_0}^{\infty}B^{\alpha}\sum\limits_{\vec{n}\in\Omega^\circ_{\text{D}}}{|c_{\vec{n}}|^k}}\\
&=& \displaystyle{\sum\limits_{k=0}^{s_0-1}{|m_k|^{\alpha}}} + \displaystyle{\sum\limits_{\vec{n}\in\Omega^\circ_{\text{D}}}B^{\alpha}\frac{|c_{\vec{n}}|^{s_0}}{1-|c_{\vec{n}}|}} < \infty.
\end{eqnarray*}
$\sum\limits_{k=0}^{s_0-1}{|m_k|^{\alpha}}$ can take an arbitrarily small value, if one let the value of $\max_{\vec{n}\in\Omega_{\text{D}}}{a_{\vec{n}}}$ be small so that the inequality above implies that $\sum^\infty_{k=0}{|m_k|^\alpha}$ can take an arbitrarily small value. (if $B$ is sufficiently small.)
Thus, $\sum^\infty_{k=0}{|m_k|^\alpha}<1$ with sufficiently small initial condition in \eqref{dlheq} and the statement of theorem 2 holds by lemma 3.1.
This completes the proof of the main theorem.
\section*{Acknowledgements}
The author is deeply grateful to Prof. Tetsuji Tokihiro who provided helpful comments and suggestions. 
The author would also like to thank to Prof. Ralph Willox and Dr. Shinsuke Iwao for useful comments and warm encouragements.


\begin{thebibliography}{99}
\bibitem{F}H.~Fujita,\ On the blowing up of solutions of the Cauchy problem for $u_t=\Delta u+u^{1+\alpha}$,\ \textit{J.\ Fac.\ Sci.\ Univ.\ Tokyo Sect.\ A Math.} 16(1966),\ 109--124.
\bibitem{H}K.~Hayakawa,\ On nonexistence of global solutions of some semilinear parabolic equations,\ \textit{Proc.\ Japan Acad.}\ 49(1973),\ 503--505.
\bibitem{KST}K.~Kobayashi, T.~Sirao and H.~Tanaka,\ On the growing up problem for semilinear heat equations,\ \textit{J.\ Math.\ Soc.\ Japan}\ 29(1977),\ 407--424.
\bibitem{W}F.~B.~Weissler,\ Existence and nonexistence of global solutions for a semilinear heat equation,\ \textit{Israel J.\ Math.} 38(1981),\ 29--40.
\bibitem{MT}K.~Matsuya and T.~Tokihiro,\ Existence and non-existence of global solutions for a discrete semilinear heat equation, \textit {Discrete Contin.\ Dynam.\ Systems} 31(2011),\ 209--220.
\end{thebibliography}
\end{document}